\documentclass{amsart}
\usepackage{geometry}                
\usepackage{graphicx}
\usepackage{amssymb}
\usepackage{epstopdf,amsmath}
\newtheorem{prop}{Proposition}
\newtheorem{lemma}[prop]{Lemma}
\newtheorem{cor}[prop]{Corollary}

\newtheorem{theorem}[prop]{Theorem}

\DeclareMathOperator{\Poly}{Poly}

\newcommand{\eps}{\epsilon}

\newcommand{\RR}{\mathbb{R}}

\title{A short proof of the multilinear Kakeya inequality}

\author{Larry Guth}

\begin{document}

\begin{abstract} We give a short proof of a slightly weaker version of the multilinear Kakeya inequality proven by Bennett, Carbery, and Tao.  \end{abstract}

\maketitle

The multilinear Kakeya inequality is a geometric estimate about the overlap pattern of cylindrical tubes in $\RR^n$ pointing in different directions.  This estimate was first proven by Bennett, Carbery, and Tao in \cite{BCT}.  Recently it has had some striking applications in harmonic analysis.  Here is a short list of some applications.  In \cite{BCT}, it was applied to prove a multilinear restriction estimate.  In \cite{BG}, Bourgain and the author used this multilinear restriction estimate to make some progress on the original restriction problem, posed by Stein in \cite{Ste}.   In \cite{B}, Bourgain used it to prove new estimates for eigenfunctions of the Laplacian on flat tori.  Most recently, in \cite{BD}, Bourgain and Demeter used the multilinear restriction estimate to prove the $l^2$ decoupling conjecture.  As a corollary of their main result, they proved essentially sharp Strichartz estimates for the Schrodinger equation on flat tori.

The goal of this paper is to give a short proof of the multilinear Kakeya inequality.  The original proof of \cite{BCT} used monotonicity properties of heat flow and it is morally based on multiscale analysis.  Later there was a proof in \cite{G} using the polynomial method.  The proof we give here is based on multiscale analysis.  I think that the underlying idea is the same as in \cite{BCT}, but the argument is organized in a more concise way.

Here is the statement of the multilinear Kakeya inequality.  Suppose that $l_{j,a}$ are lines in $\RR^n$, where $j = 1, ..., n$, and where $a = 1, ..., N_j$.  We write $T_{j,a}$ for the characteristic function of the 1-neighborhood of $l_{j,a}$.  

\begin{theorem} \label{main} Suppose that $l_{j,a}$ are lines in $\RR^n$, and that each line $l_{j,a}$ makes an angle of at most $(10 n)^{-1}$ with the $x_j$-axis.

Let $Q_S$ denote any cube of side length $S$.  Then for any $\eps > 0$ and any $S \ge 1$, the following integral inequality holds:

\begin{equation} \label{multkak}  \int_{Q_S} \prod_{j=1}^n \left( \sum_{a = 1}^{N_j} T_{j,a} \right)^{\frac{1}{n-1}}  \le C_\eps S^\eps \prod_{j=1}^n N_j^{\frac{1}{n-1}}. \end{equation}
\end{theorem}

Theorem \ref{main} is slightly weaker than the estimates proven in \cite{BCT} or \cite{G}.  The strongest form of the estimate was proven in \cite{G}: equation \ref{multkak} holds without the factor $S^\eps$.  However, Theorem \ref{main} is exactly what \cite{BCT} use to prove the multilinear restriction theorem, Theorem 1.16 in \cite{BCT}.  (See Proposition 2.1 and the following paragraph in \cite{BCT}.)  

There is a similar geometric setup in harmonic analysis papers by Bennett and Bez \cite{BB} and Bejenaru, Herr, and Tataru \cite{BHT}.   There are also some related ideas in recent work by Cs{\H o}rnyei and Jones \cite{CJ} in geometric measure theory, connected with the tangent spaces of nullsets and the sets of non-differentiability of Lipschitz functions.  I heard Marianna Cs{\H o}rnyei give a talk about their work, and their approach helped suggest the argument in this paper.

\section{The proof of the multilinear Kakeya inequality}

In this section, we prove Theorem \ref{main}.

\subsection{Reduction to nearly axis parallel tubes}  \label{subsecreducsmallangle}

The first observation of \cite{BCT} is that it suffices to prove this type of estimate when the angle $(10 n)^{-1}$ is replaced by an extremely small angle $\delta$.   More precisely, Theorem \ref{main} follows from 

\begin{theorem} \label{maindelta} For every $\eps > 0$, there is some $\delta > 0$ so that the following holds.  

Suppose that $l_{j,a}$ are lines in $\RR^n$, and that each line $l_{j,a}$ makes an angle of at most $\delta$ with the $x_j$-axis.  Then for any $S \ge 1$ and any cube $Q_S$ of side length $S$, the following integral inequality holds:

\begin{equation} \label{multkakdelta} \int_{Q_S} \prod_{j=1}^n \left( \sum_{a = 1}^{N_j}  T_{j,a} \right)^{\frac{1}{n-1}}  \lesssim_\eps S^\eps \prod_{j=1}^n N_j^{\frac{1}{n-1}}. 
\end{equation}

\end{theorem}

We explain how Theorem \ref{maindelta} implies Theorem \ref{main}.  Let $e_j$ be the unit vector in the $x_j$ direction, and let $S_j \subset S^{n-1}$ be a spherical cap around the point $e_j$ of radius $(10 n)^{-1}$.  By the hypotheses of Theorem \ref{main}, every line in $l_{j,a}$ has direction in the cap $S_j$.  

Given $\eps > 0$, we choose $\delta > 0$ as in Theorem \ref{maindelta}.

We subdivide the cap $S_j$ into smaller caps $S_{j, \beta}$ of radius $\delta/10$.  The number of caps $S_{j, \beta}$ is at most $\Poly(\delta^{-1}) \lesssim_\eps 1$.  We can break the left-hand side of Equation \ref{multkakdelta} into contributions from different caps $S_{j, \beta}$.  We write $l_{j,a} \in S_{j, \beta}$ if the direction of $l_{j,a}$ lies in $S_{j, \beta}$.  Since the number of caps is $\lesssim_\eps 1$, we get:

\begin{equation}  \label{decompbeta}  \int_{Q_S} \prod_{j=1}^n \left( \sum_{a = 1}^{N_j}  T_{j,a} \right)^{\frac{1}{n-1}} \lesssim_\eps 
\sum_{\beta_1, ..., \beta_n} \int_{Q_S} \prod_{j=1}^n \left( \sum_{l_{j,a} \in S_{j, \beta_j}}  T_{j,a} \right)^{\frac{1}{n-1}}. \end{equation}

We claim that each term on the right hand side of Equation \ref{decompbeta} is controlled by Theorem \ref{maindelta}.

If we choose $\beta_j$ so that $S_{j, \beta_j}$ contains $e_j$, then Theorem 2 directly applies.  If not, we have to make a linear change of coordinates, mapping the center of $S_{j, \beta_j}$ to $e_j$.  The condition that the angle between $l_{j,a}$ and $e_j$ is at most $(10 n)^{-1}$ guarantees that this linear change of coordinates distorts lengths by at most a factor of 2 and distorts volumes by at most a factor of $2^n$.  In the new coordinates, the integral is controlled by Theorem 2.

\subsection{The axis parallel case (Loomis-Whitney)}  \label{subsecaxispar}

We have just seen that the multilinear Kakeya inequality reduces to a nearly axis-parallel case.  Next we consider the exactly axis-parallel case: the case that $l_{j,a}$ is parallel to the $x_j$-axis.  In this axis-parallel case, the multilinear Kakeya inequality follows immediately from the Loomis-Whitney inequality, proven in \cite{LW}.  To state their result, we need a little notation.

Let $\pi_j: \RR^n \rightarrow \RR^{n-1}$ be the linear map that forgets the $j^{th}$ coordinate: 

$$ \pi_j( x_1, ..., x_n) = (x_1, ..., x_{j-1}, x_{j+1}, ..., x_n). $$

\begin{theorem} (Loomis-Whitney)  Suppose that $f_j: \RR^{n-1} \rightarrow \RR$ are (measurable) functions.  Then the following integral inequality holds:

$$ \int_{\RR^n} \prod_{j=1}^n f_j \left( \pi_j(x) \right)^\frac{1}{n-1} \le \prod_{j=1}^n \| f_j \|_{L^1(\RR^{n-1})}^{\frac{1}{n-1}}. $$

\end{theorem}

If the line $l_{j,a}$ is parallel to the $x_j$-axis, then it can be defined by writing $\pi_j(x) = y_a$ for some $y_a \in \RR^{n-1}$.  Then the function $\sum_{a} T_{j,a}(x)$ is equal to $\sum_{a} \chi_{B(y_a, 1)} (\pi_j (x))$.  We apply the Loomis-Whitney inequality with $f_j = \sum_{a} \chi_{B(y_a, 1)}$, which gives $\| f_j \|_{L^1(\RR^{n-1})} = \omega_{n-1} N_j \sim N_j$.  We get:

$$ \int_{\RR^n} \prod_{j=1}^n \left( \sum_{a = 1}^{N_j}  T_{j,a} \right)^{\frac{1}{n-1}} =
\int_{\RR^n} \prod_{j=1}^n f_j (\pi_j(x))^{\frac{1}{n-1}} \le \prod_{j=1}^n \| f_j \|_{L^1(\RR^{n-1})} \sim \prod_{j=1}^n N_j^{\frac{1}{n-1}}. $$

\subsection{The multiscale argument}

In this subsection, we will prove Theorem \ref{maindelta} - the multilinear Kakeya inequality for tubes that make a tiny angle with the coordinate axes.  For any $\eps > 0$, we get to choose $\delta > 0$, and we know that each line $l_{j,a}$ makes an angle of at most $\delta$ with the $x_j$-axis.  
We have seen how to prove the inequality for lines that are exactly parallel to the axes.  We just have to understand how to control the effect of a tiny tilt in the tubes.  

The main idea of the argument is that instead of trying to directly perform an estimate at scale $S$, we work at a sequence of scales $\delta^{-1}, \delta^{-2}, \delta^{-3}, $ etc.  up to an arbitrary scale $S$.  To get from each scale to the next scale, we use the Loomis-Whitney inequality.

To set up our multiscale argument, we use not just tubes of radius 1 but tubes of multiple scales.  We define

$$ T_{j,a, W} := \textrm{the characteristic function of the $W$-neighborhood of } l_{j,a}. $$

$$ f_{j, W} := \sum_{a = 1}^{N_j} T_{j, a , W}. $$

The key step in the argument is the following lemma relating one scale to a scale $\delta^{-1}$ times larger:

\begin{lemma} \label{mainlemma} Suppose that $l_{j,a}$ are lines with angle at most $\delta$ from the $x_j$ axis.  Let $T_{j, a, W}$ and $f_{j, W}$ be as above.  

If $S \ge \delta^{-1} W$, and if $Q_S$ is any cube of sidelength $S$, then

$$ \int_{Q_S} \prod_{j=1}^n f_{j, W}^\frac{1}{n-1} \le C_n \delta^n \int_{Q_S} \prod_{j=1}^n f_{j, \delta^{-1} W}^\frac{1}{n-1}. $$

\end{lemma}

\begin{proof} We divide $Q_S$ into subcubes $Q$ of side length between $\frac{1}{20n} \delta^{-1} W$ and $\frac{1}{10n} \delta^{-1} W$.  For each such cube $Q$, it suffices to prove that

$$ \int_{Q} \prod_{j=1}^n f_{j, W}^\frac{1}{n-1} \le C_n \delta^n \int_{Q} \prod_{j=1}^n f_{j, \delta^{-1} W}^\frac{1}{n-1}. $$

Because the side length of $Q$ is $\le (1/10n) \delta^{-1} W$, the intersection of $T_{j, a, W}$ with $Q$ looks fairly similar to an axis-parallel tube.  More precisely, for each $j,a$ there is an axis-parallel tube $\tilde T_{j,a, 2W}$, of radius $2 W$, so that for $x \in Q$, $T_{j, a, W} (x) \le \tilde T_{j, a, 2 W}(x)$.  Therefore, 

$$ \int_{Q} \prod_{j=1}^n f_{j, W}^\frac{1}{n-1} = \int_{Q} \prod_{j=1}^n \left( \sum_a T_{j, a, W} \right)^\frac{1}{n-1} \le   \int_{Q} \prod_{j=1}^n \left( \sum_a \tilde T_{j, a, 2W} \right)^\frac{1}{n-1} . $$

This last integral involves axis-parallel tubes and we can bound it using the Loomis-Whitney inequality, as in Subsection \ref{subsecaxispar}.  Let $N_j(Q)$ be the number of tubes $T_{j, a, W}$ that intersect $Q$.  By Loomis-Whitney, we get 

$$  \int_{Q} \prod_{j=1}^n \left( \sum_a \tilde T_{j, a, 2W} \right)^\frac{1}{n-1} \le C_n W^n \prod_{j=1}^n N_j(Q)^\frac{1}{n-1}. $$

Since the sidelength of $Q$ is at most $(1/ 10n) \delta^{-1} W$, the diameter of $Q$ is at most $(1/10) \delta^{-1} W$.  If $T_{j, a, W}$ intersects $Q$, then $T_{j, a, \delta^{-1} W}$ is identically 1 on $Q$.  Therefore,

$$ C_n W^n \prod_{j=1}^n N_j(Q)^\frac{1}{n-1} \le C_n W^n \oint_Q \prod_{j=1}^n \left( \sum_{a} T_{j, a, \delta^{-1} W} \right)^\frac{1}{n-1}  = C_n \delta^n \int_Q \prod_{j=1}^n \left( \sum_{a} T_{j, a, \delta^{-1} W} \right)^\frac{1}{n-1}. $$

This finishes the proof of the lemma.   \end{proof}

Using Lemma \ref{mainlemma}, we can now prove Theorem \ref{maindelta}.
Suppose first that $S$ is a power $S = \delta^{-M}$, for an integer $M \ge 0$ .  Using Lemma \ref{mainlemma} repeatedly we get:

$$ \int_{Q_S} \prod_{j=1}^n \left( \sum_{a} T_{j,a} \right)^\frac{1}{n-1} = \int_{Q_S} \prod_{j=1}^n f_{j, 1}^\frac{1}{n-1} \le C_n \delta^n \int_{Q_S} \prod_{j=1}^n f_{j, \delta^{-1}}^\frac{1}{n-1} \le $$

$$ \le C_n^2 \delta^{2n} \int_{Q_S} \prod_{j=1}^n f_{j, \delta^{-2}}^\frac{1}{n-1} \le ...
\le $$

$$ \le C_n^M \delta^{Mn} \int_{Q_S} \prod_{j=1}^n f_{j, \delta^{-M}}^\frac{1}{n-1} = C_n^M \oint_{Q_S} \prod_{j=1}^n f_{j, \delta^{-M}}^\frac{1}{n-1} . $$

We know that $f_{j, \delta^{-M}} \le N_j$, and so we get

$$  \int_{Q_S} \prod_{j=1}^n \left( \sum_{a} T_{j,a} \right)^\frac{1}{n-1} \le C_n^M \prod_{j=1}^n N_j^\frac{1}{n-1}. $$

Since $S = \delta^{-M}$, we see that $M = \frac{ \log S} { \log \delta^{-1}}$.  Therefore,

$$ C_n^M = S^{\frac{C_n}{\log \delta^{-1}}}. $$

Now we choose $\delta > 0$ sufficiently small so that $\frac{C_n}{\log \delta^{-1}} \le \eps$.  For $S = \delta^{-M}$, we have now proven that

$$  \int_{Q_S} \prod_{j=1}^n \left( \sum_{a} T_{j,a} \right)^\frac{1}{n-1} \le S^\eps \prod_{j=1}^n N_j^\frac{1}{n-1}. $$

Finally, for an arbitrary $S \ge 1$, we can find an integer $M \ge 0$ and cover $Q_S$ with $C(\delta)$ cubes of side length $\delta^{-M}$.  Therefore, for any $S \ge 1$, we see

$$  \int_{Q_S} \prod_{j=1}^n \left( \sum_{a} T_{j,a} \right)^\frac{1}{n-1} \le C_\eps S^\eps \prod_{j=1}^n N_j^\frac{1}{n-1}. $$

This finishes the proof of Theorem \ref{maindelta} and hence the proof of Theorem \ref{main}.  

\section{Some small generalizations}

In this section, we mention two minor generalizations of Theorem \ref{main} that can be useful in applications.  

One minor generalization is to add weights.

\begin{cor} \label{weight} Suppose that $l_{j,a}$ are lines in $\RR^n$, and that each line $l_{j,a}$ makes an angle of at most $(10 n)^{-1}$ with the $x_j$-axis.  Suppose that $w_{j,a} \ge 0$ are numbers.  Let $T_{j,a}$ be the characteristic function of the 1-neighborhood of $l_{j,a}$.  Define

$$ f_j := \sum_{a} w_{j,a} T_{j,a}. $$

Let $Q_S$ denote any cube of side length $S$.  Then for any $\eps > 0$ and any $S \ge 1$, the following integral inequality holds:

\begin{equation} \label{multkakweight}  \int_{Q_S} \prod_{j=1}^n f_j^{\frac{1}{n-1}}  \le C_\eps S^\eps \prod_{j=1}^n \left( \sum_a w_{j,a} \right)^{\frac{1}{n-1}}. \end{equation}
\end{cor}

\begin{proof} If the weights $w_{j,a}$ are positive integers, the result follows from Theorem \ref{main} by including each tube multiple times.  (The tubes $T_{j,a}$ in Theorem \ref{main} do not need to be distinct.)  Then by scaling the theorem holds for rational weights and by continuity for real weights.
\end{proof}

In Theorem \ref{main}, we assumed that $l_{j,a}$ makes an angle less than $(10n)^{-1}$ with the $x_j$-axis.  This was simple to state, but it is not the most general condition we could make about the angles of the lines $l_{j,a}$.  Here is a more general setup that can be useful in applications.

\begin{cor}
Suppose that $S_j \subset S^{n-1}$.  Suppose that $l_{j,a}$ are lines in $\RR^n$ and that the direction of $l_{j,a}$ lies in $S_j$.  Suppose that for any vectors $v_j \in S_j$,

$$ | v_1 \wedge ... \wedge v_n | \ge \nu.$$

\noindent Let $T_{j,a}$ be the characteristic function of the 1-neighborhood of $l_{j,a}$.

Let $Q_S$ denote any cube of side length $S$.  Then for any $\eps > 0$ and any $S \ge 1$, the following integral inequality holds:

\begin{equation}  \int_{Q_S} \prod_{j=1}^n \left( \sum_{a = 1}^{N_j} T_{j,a} \right)^{\frac{1}{n-1}}  \le C(\eps) \Poly(\nu^{-1}) S^\eps \prod_{j=1}^n N_j^{\frac{1}{n-1}}. \end{equation}

\end{cor}

\begin{proof} This estimate follows from Theorem \ref{maindelta} by essentially the same argument as in subsection \ref{subsecreducsmallangle}.  We again cover $S_j$ by caps $S_{j, \beta}$ of a small radius $\rho$.  As long as $\rho \le \frac{1}{100n} \nu$, we can guarantee that $|v_1 \wedge ... \wedge v_n |  \ge \nu/2$ for all $v_j \in S_{j, \beta}$.  We pick a sequence of caps $S_{1, \beta_1}, ..., S_{n, \beta_n}$.  We change coordinates so that the center of the cap $S_{j, \beta_j}$ is mapped to the coordinate unit vector $e_j$.  The distortion of lengths and volumes caused by the coordinate change is $\Poly(\nu^{-1})$.  
We apply Theorem \ref{maindelta} in the new coordinates.  If $\rho = \rho(\eps)$ is small enough, the image of $S_{j,\beta}$ is contained in a cap of radius $\delta = \delta(\eps)$ as in Theorem \ref{maindelta} -- and this gives the desired estimate with error factor $C_\eps \Poly(\nu^{-1}) S^\eps$.  Finally, we have to sum over $C(\eps) \Poly(\nu^{-1})$ different choices of $S_{1, \beta_1} ... S_{n, \beta_n}$.  
\end{proof}

This is not the sharpest known estimate when it comes to the dependence on $\nu$.  See \cite{BCT} or \cite{G} for sharper estimates.  However, in the applications I know, this estimate is sufficient.

\section{On Lipschitz curves}

The proof of Theorem \ref{maindelta} applies not just to straight lines, but also to Lipschitz curves with Lipschitz constant at most $\delta$.  

Suppose that $g_{j, a}: \RR \rightarrow \RR^{n-1}$ is a Lipschitz function with Lipschitz constant at most $\delta$.  We let $\gamma_{j,a}$ be the graph given by $(x_1, ..., x_{j-1}, x_{j+1}, ... , x_n) = g_{j,a}(x_j)$.  The curves $\gamma_{j,a}$ will play the role of the lines $l_{j,a}$ -- lines are the special case that the functions $g_{j,a}$ are affine.  We let $T_{j,a}$ be the characteristic function of the 1-neighborhood of $\gamma_{j,a}$.  With this setup, Theorem \ref{maindelta} generalizes to Lipschitz curves with the same proof.

\begin{theorem}
 For every $\eps > 0$, there is some $\delta > 0$ so that the following holds.  

Suppose that $\gamma_{j,a}$ are as above: Lipschitz curves in $\RR^n$ with an angle of at most $\delta$ with the $x_j$-axis.  Then for any $S \ge 1$ and any cube $Q_S$ of side length $S$, the following integral inequality holds:

\begin{equation} \label{multkakdeltalip} \int_{Q_S} \prod_{j=1}^n \left( \sum_{a = 1}^{N_j}  T_{j,a} \right)^{\frac{1}{n-1}}  \lesssim_\eps S^\eps \prod_{j=1}^n N_j^{\frac{1}{n-1}}. 
\end{equation}
\end{theorem}

This multilinear estimate for Lipschitz curves has a similar flavor to some estimates of Cs{\H o}rnyei and Jones described in \cite{CJ}.  


\end{document}